\documentclass[11pt, oneside]{article}   	
\usepackage{geometry}                		
\usepackage{geometry, amsmath, amsfonts, amssymb}                		
\usepackage{url}
\usepackage[parfill]{parskip}    		
\usepackage{graphicx}				
\usepackage{amssymb}
\usepackage{fancyhdr}
\usepackage{leftidx}
\usepackage{amsthm}
\usepackage{mathrsfs}
\usepackage{theoremref}
\usepackage{longtable}
\usepackage{amsmath}
\usepackage{graphicx}
\usepackage{placeins}
\title{Union-Free Families of Subsets}
\author{Andy Loo\\aloo@princeton.edu\\Department of Mathematics, Princeton University}

\def\bm{\begin{bmatrix}}
\def\endm{\end{bmatrix}}

\newcommand{\ceil}[1]{\lceil#1\rceil}

\def\hat{\widehat}

\theoremstyle{note}

\newtheorem{thm}{Theorem}[section]
\newtheorem{lem}[thm]{Lemma}
\newtheorem{defn}[thm]{Definition}
\newtheorem{conj}[thm]{Conjecture}
\newtheorem{eg}[thm]{Example}

\makeatletter
\def\thm@space@setup{%
  \thm@preskip=\parskip \thm@postskip=0pt
}
\makeatother

\begin{document}
\maketitle

\begin{abstract}
This paper discusses the question of how many non-empty subsets of the set $[n] = \{ 1, 2, ..., n\}$  we can choose so that no chosen subset is the union of some other chosen subsets. Let $M(n)$ be the maximum number of subsets we can choose. 
We construct a series of such families, which leads to lower bounds on $M(n)$. We also give upper bounds on $M(n)$. Finally, we propose several conjectures on the tightness of our lower bound for $M(n)$. 
\end{abstract}

\section*{Keywords}

Subset, union, antichain, LYM inequality, Sperner's theorem

\section{Introduction}

Suppose a legislator wishes to filibuster against a bill by proposing a large number of amendments to it. 
The bill consists of $n$ articles, where $n$ is a positive integer. 
Suppose each amendment must be in the form of  repealing one or more articles of the bill, and that if two or more amendments are passed, the effect is to repeal  the union of the subsets of articles that the passed amendments seek to repeal. 
At most how many amendments can the legislator propose without any amendment being superfluous (and hence ruled out of order by the chair)? 

Here, an amendment is said to be superfluous if its removal does not alter the range of possible effects that can be achieved by passing various combinations of amendments.

A different question in this context is at most how many amendments can be proposed such that no two different \emph{collections} of amendments would achieve the same effect. If there are $N$ amendments, there are $2^N$ collections of amendments (including the empty collection). 
We require that each collection of amendments repeal a different set of articles. 
As there are only $2^n$ sets of articles, we must have $2^N \leq 2^n$ and hence $N \leq n$. Indeed, one can propose $n$ distinct amendments where each consists of repealing exactly one article. Therefore the answer to this question is $n$.
However, this is not the question that we will discuss in this paper. Instead, we will consider the formulation stated in the preceding paragraphs and require that no \emph{single} amendment be superfluous in its own right. 

\section{Preliminaries}

We define the following: 
\begin{itemize} 
\item For any positive integer $n$, let $[n] = \{ 1, 2, ..., n\}$.
\item Let $[a, b]  =\{ a, a+1, ..., b\}$ for integers $a \leq b$, and $[a, b] = \emptyset$ for $a > b$.
\item For any set $S$, let $|S|$ denote the cardinality of $S$. 
\item For any set $S$, let $P(S)$ be the power set of $S$. 
\item For any set $S$ and integer $k$, let $\binom{S}{k}$ be the family of subsets of $S$ of size $k$. 
\item For any family $F$ of {non-empty} subsets of $[n]$, let 
\[ U(F) = \left \{  \bigcup_{A \in I} A : I \subseteq F  \right \}.\]
i.e. let $U(F)$ be the collection of unions of subsets of $[n]$ that $F$ contains. 
\item A finite family $F$ of sets is said to be \emph{union-free} if there does not exist any set $A \in F$ that is the union of some other sets $A_1, A_2, ..., A_k \in F$. 
Also, for any set $S$, let $UF(S)$ denote the collection of all union-free families of subsets (possibly including the empty set) of $S$, and let $M(n)$ be the maximum size of a union-free family of \emph{non-empty} subsets of $S$. 
\item For any two families $F_1, F_2$ of sets, define 
\[ F_1 \oplus F_2 = \{ A_1 \cup A_2 : A_1 \in F_1, A_2 \in F_2 \}.\]

\end{itemize} 

\section{An Equivalent Interpretation}

In the filibuster problem for a bill with $n$ articles, each amendment corresponds to a \emph{non-empty} subset of $[n]$. 
The question is to find the maximum size of a family $F$ of subsets of $[n]$ such that there does not exist any subset $A \in F$ for which 
\[ U(F) = U(F \setminus \{ A\}).\]

\begin{thm}
For any family $F$ of non-empty subsets of $[n]$, the following conditions are equivalent: 

(i) There does not exist any subset $A \in F$ for which $U(F) = U(F \setminus \{ A\})$.

(ii) $F$ is union-free.
\end{thm}

\begin{proof}
If condition (ii) does not hold, i.e. if a subset $A \in F$ is the union of some other subsets $A_1, A_2, ..., A_k \in F$, then for any $G \subseteq F$, we have 
\[ A \cup \bigcup_{B \in G} B = \bigcup_{i=1}^k A_i \cup \bigcup_{B \in G} B \in U(F \setminus \{ A\}).\]
Hence $U(F) = U(F \setminus \{ A\})$, violating condition (i). So we have that condition  (i) implies condition (ii).

Conversely, if condition (i) does not hold, i.e. if there exists a subset $A \in F$ for which $U(F) = U(F \setminus \{A\})$, then since $A \in U(F)$, we have $A \in U(F \setminus \{ A\})$, so there exists some subsets $A_1, A_2, ..., A_k \in F\setminus \{ A\}$ such that $A = \bigcup_{i=1}^k A_i$, violating condition (ii). Therefore condition (ii) implies condition (i). 
\end{proof}

Thus, the question is reduced to finding the maximum size of a union-free family of subsets of $[n]$. 

\section{Antichains} 

If $A = \bigcup_{i=1}^k A_i$, then $A_i \subseteq A$ for $i = 1, 2, ..., k$. So in order to have a union-free family $F$ of subsets, a sufficient condition is that there do not exist two distinct  subsets  $A, B$ in the family $F$  with $B \subseteq A$, i.e. $F$ is an antichain. 

We recall the following results: 

\begin{thm} [LYM inequality\footnote{Lubell-Yamamoto-Meshalkin inequality.}]
For any antichain $F \subseteq P([n])$, the following inequality holds:  
\[ \sum_{A \in F} \frac{1}{\binom{n}{|A|}} \leq 1.\]
\end{thm}

\begin{thm}[Sperner]
When $n$ is even, the only largest antichain of $[n]$ is 
\[ \binom{[n]}{n/2}.\]
When $n$ is odd, the only largest antichains of $[n]$ are 
\[ \binom{[n]}{\ceil{n/2}} \text{ and } \binom{[n]}{\lfloor n/2 \rfloor}.\]
\end{thm}

\section{An Improvement}

In the previous section we showed that $\binom{[n]}{\ceil{n/2}}$ is one of the largest antichains in $[n]$.
It is possible to add subsets to this antichain so that the resulting family is still union-free. 

\begin{defn}
For  positive integers $n \geq m_1 > m_2 > \cdots > m_l$, let 
$q(n; m_1; m_2; ...; m_l)$ be the family 
\[\binom{[n]}{m_1} \cup \binom{[m_1 - 1]}{m_2} \cup \cdots  \cup \binom{[m_{l-1} - 1]}{m_l}.\]
\end{defn}

\begin{defn}
For  positive integers $n$, let 
\[ Q(n) = \left \{ q(n; m_1; m_2; ...; m_l) : n \geq m_1 > m_2 > \cdots  > m_{l-1} > m_l = 1 \right \}.\]
\end{defn}

For example, 
\[ q(20; 15; 6; 1) = \binom{[20]}{15} \cup \binom{[14]}{6} \cup \binom{[5]}{1} \in Q(20).\]

In particular, we define the following: 

\begin{defn}
For  positive integers $n$, let 
\[ q(n) = q(n; m_1; m_2; ...; m_l) ,\]
where $m_1 = \ceil{n/2}$, $m_j = \ceil{\frac{m_{j-1}}{2}}$ for  $j = 2, ..., l$, and $m_l = 1$. 
\end{defn}

For example, 
\[ q(20) = \binom{[20]}{10} \cup \binom{[9]}{5} \cup \binom{[4]}{2} \cup \binom {[1]}{1}.\]
Note that the maximal antichain $\binom{[n]}{\ceil{n/2}}$ is contained in $q(n)$. 

\begin{thm}
For any positive integer $n$, any family in $Q(n)$ is union-free.
\end{thm}

\begin{proof}
For any $F \in Q(n)$, let $F = q(n; m_1; m_2; ...; m_l)$. 
Suppose for the sake of contradiction that there is some $A \in F$ that is the union of some other subsets $A_1, A_2, ..., A_k \in F$. 
We have $\max \{ |A_1|, ..., |A_k| \} = m_j$ for some $1 \leq j \leq l$. 
Then $A_i \subseteq [m_{j-1} - 1]$ for all $i$, and so $\bigcup_{i=1}^k A_i \subseteq [m_{j-1} - 1]$. 
Since $|A_i|< |A|$ for all $i$, we have
$\left | \bigcup_{i=1}^k A_i \right |  \leq  m_{j-1} - 1 < m_{j-1} \leq |A|$, a contradiction. 
\end{proof}

\begin{defn}
A union-free family $F$ of subsets of $[n]$ is said to be \emph{maximal} if it is not a proper sub-family of a union-free family of subsets of $[n]$. 
\end{defn}

Clearly, this is equivalent to saying that no other subset of $[n]$ can be added to $F$ so that the resulting family is still union-free.  

We want to prove that any family in $Q(n)$ is maximal. To do so, we will need a lemma.

\begin{defn}
We say that a family $F$ of subsets of $[n]$ can \emph{augment} a set $S$ to a set $T$ if there exist subsets $A_1, A_2, ..., A_k \in F$ such that $S \cup \bigcup_{i=1}^k A_i = T$. 
\end{defn}

\begin{lem} \thlabel{up down augment}
Given any $F \in Q(n)$, for any integers $s, t$ with $0 \leq s < t \leq n$, $F$ can augment any subset of $[n]$ of size $s$ to some subset of $[n]$ of size $t$.
\end{lem}

\begin{proof}
We proceed by induction on $n$. The statement is trivial for $n =1$.

Assume the statement is true for all positive integers less than $n$. Consider the statement for $n$.
Given any $F \in Q(n)$, let $F = q(n; m_1; m_2; ...; m_l)$.
Note that for any integers $s, t$ with $0 \leq s < m_1 \leq t$, $\binom{[n]}{m_1}$ can augment any subset of size $s$ to some subset of size $t$. 
Thus we only require that for any integers $s, t$ with $0 < s < t \leq m_1 - 1$, any family in $Q(m_1 - 1)$ can augment any subset of $[n]$ of size $s$ to some subset of $[n]$ of size $t$. 
This is true because by the inductive hypothesis, any family in $Q(m_1 - 1)$ can augment any subset of $[m_1 - 1]$ of size $s$ to some subset of $[m_1 - 1]$ of size $t$. 
\end{proof}

\begin{thm} 
For any positive integer $n$, any family in $Q(n)$ is a maximal union-free family of subsets of $[n]$.
\end{thm}

\begin{proof}
We proceed by induction on $n$. The statement is trivial for $n= 1$. 

Assume the statement is true for all positive integers less than $n$. Consider the statement for $n$. 
Given any $F \in [n]$,   let $F = q(n; m_1; m_2; ...; m_l)$.
Contemplate adding a subset $S$ of $[n]$ to $F$, where $S \notin F$. 
If $|S| > m_1$, then $S$ is the union of some of the subsets in $\binom{[n]}{m_1}$. 
Note that $|S| \ne m_1$ because $\binom{[n]}{m_1} \subseteq F$. 
For $|S| < m_1$, if $S \subseteq [m_1 - 1]$, the inductive hypothesis implies that the family after adding $S$ is not union-free. 

So the only case remaining is that $|S| \leq m_1 - 1$ and $S \not \subseteq [m_1 - 1]$. 
Thus
\[ |S \setminus [m_1 - 1] | \geq 1 \text{ and } |S \cap [m_1 - 1] | \leq m_1 - 2. \]
By \thref{up down augment}, $q(m_1;m_2;...;m_l)$ can augment $|S \cap [m_1 - 1]$ to a subset of $[m_1- 1]$ of size $m_1 -  |S \setminus [m_1 - 1] |$.
That is, $q(m_1; m_2 ; ...; m_l)$ can augment $S$ to a subset of $[n]$ of size $m_1$, which is in $F$. 
\end{proof}


Recall that $q(n)$ is merely a special case in $Q(n)$. 
Which family in $Q(n)$, we may ask, is largest? Is it $q(n)$? 
By taking $m_1 > \ceil{n/2}$, we may increase $l$ (i.e. have more terms), but we will suffer from the fact that the largest term is smaller than $\binom{n}{\ceil{n/2}}$.

\begin{lem} \thlabel{stirling}
For integers $k, j$ with $0 \leq j \leq k$, 
\[ \binom{k}{j} \simeq \frac{1}{\sqrt{2 \pi}} \cdot \frac{k^{k + 1/2}}{j^{j + 1/2} (k-j)^{k - j + 1/2}}.\]
\end{lem}

\begin{proof}
Stirling's approximation gives 
\[ n! \simeq \sqrt{2 \pi} \cdot n^{n + 1/2}\cdot e^{-n} \]
for positive integers $n$. 
We have 
\[ \binom{k}{j} \simeq \frac{\sqrt{2 \pi}k^{k + 1/2} e^{-k} }{\sqrt{2 \pi}j^{j+1/2} e^{-j}  \sqrt{2 \pi} (k-j)^{k - j+1/2} e^{-(k-j)}}
=\frac{1}{\sqrt{2 \pi}} \cdot \frac{k^{k + 1/2}}{j^{j + 1/2} (k-j)^{k -j + 1/2}}\]
as desired. 
\end{proof}

We can approximate $\binom{n}{\ceil{n/2}}$ by
\[ \frac{1}{\sqrt{2 \pi}} \cdot \frac{n^{n + 1/2}}{(n/2)^{n/2 + 1/2} (n/2)^{n/2 + 1/2}} = \frac{1}{\sqrt{2 \pi}}\cdot \frac{n^{n + 1/2}}{(n/2)^{n+1}} = \sqrt{\frac{2}{\pi}} \cdot \frac{2^n}{n^{1/2}}.\]
So in the sum 
\[ \binom{n}{\ceil{n/2}} + \binom{\ceil{n/2}-1}{\ceil{\frac{\ceil{n/2}-1}{2}}} + \cdots \]
we have 
\[ \frac{\binom{\ceil{n/2}-1}{\ceil{\frac{\ceil{n/2}-1}{2}}}}{\binom{n}{\ceil{n/2}}}
\simeq \frac{2^n}{n^{1/2}} \cdot \frac{(n/2)^{1/2}}{2^{n/2}} = 2^{(n-1)/2}.\]
Thus the term $\binom{n}{\ceil{n/2}}$ is highly dominant. 

Roughly speaking, if a family $F \in Q(n)$ ``deviates too much'' from $q(n)$, each of its terms will be much smaller than $\binom{n}{\ceil{n/2}}$, so $F$ is likely to be smaller than $F(n)$. 
If, on the other hand, $F$ ``looks like''  $q(n)$, then the largest term in $F$ is also highly dominant, but it is smaller than $\binom{n}{\ceil{n/2}}$, so $F$ is also likely to be smaller than $F(n)$. 
This informal argument suggests that $q(n)$ is likely close to being the largest family in $Q(n)$.

\section{Cushioning}

This section is motivated by the following example:
\begin{eg}
\rm{
For $n = 5$, note that the union-free family 
\[ \left ( \binom{[4]}{2} \oplus \{ \emptyset, \{ 5\} \} \right ) \cup \{ \{ 1 \} \} \]
contains
$ 2 \binom{4}{2} + 1 = 13 $
sets, more than the family 
\[ q(5) = \binom{[5]}{3} \cup \binom{[2]}{1},\]
which contains 
$ \binom{5}{3} + \binom{2}{1} = 12$
sets.} 
\end{eg}

Apart from the above ``doubling'' technique, we can also do ``tripling":

\begin{eg}
\rm{The family 
\[ \left (\binom{[3]}{2} \oplus \{ \emptyset, \{ 4 \}, \{ 4, 5 \} \} \right ) \cup \{ \{ 1 \} \} \]
is also union-free, although it is smaller than  $q(5)$, as 
$ 3 \binom{3}{2} + 1 = 10 < 12$.
} \end{eg}

Indeed, we can define a general concept of ``cushioning'' as follows: 

\begin{defn}\thlabel{cushioning def} 
For  positive integers $n, m_1, h_1, m_2, h_2, ..., m_l, h_l$ with
\[ n \geq m_1 + h_1 \geq m_1 > m_2 + h_2 \geq m_2 > ... > m_l + h_l \geq m_l ,\]
and union-free families $F_1 \in UF([n - h_1 + 1, n])$, $F_2 \in UF([m_1 - h_2, m_1 - 1])$, ..., $F_l \in UF([m_{l-1} - h_l, m_{l-1}-1])$,
let 
\[ \begin{split}
& q(n;m_1,h_1,F_1; m_2,h_2,F_2; ... ; m_l,h_l,F_l)  \\
&=  \left ( \binom{[n-h_1]}{m_1} \oplus F_1 \right ) \cup \left ( \binom{[m_1- h_2 -1]}{m_2} \oplus F_2 \right ) \cup \cdots  \cup \left ( \binom{[m_{l-1}-h_l - 1]}{m_l} \oplus F_l \right ).
\end{split} \]
Note that we allow each $F_j$ to contain the empty set $\emptyset$. 
\end{defn}

Note that $q(n; m_1, 0, \{ \emptyset \}; m_2, 0, \{ \emptyset \}; ... ; m_l, 0, \{ \emptyset \} \}$ reduces to $q(n;m_1;m_2;...;m_l)$. 

For example, 
\[ q(5;4,0,\{ \emptyset \}; 1, 2, \{ \emptyset, \{ 2\}, \{ 3\} \}) 
= \binom{[5]}{4} \cup \{ \{ 1\} , \{ 1, 2\}, \{ 1, 3 \} \}.\]

\begin{defn}
For positive integers $n$, let $\widehat Q(n)$ be the set of all
\[ q(n;m_1,h_1,F_1; m_2,h_2,F_2; ... ; m_l,h_l,F_l)\]
satisfying the specifications in \thref{cushioning def}.
\end{defn}

\begin{thm} \thlabel{cushioning free} 
For any positive integer $n$, any family in $\widehat Q(n)$ is union-free. 
\end{thm}

\begin{proof}
For any $F \in \widehat Q(n)$, let $F = q(n;m_1,h_1,F_1; m_2,h_2,F_2; ... ; m_l,h_l,F_l)$. 
Suppose for the sake of contradiction that there is some $A \in F$ that is the union of some other subsets $A_1, A_2, ..., A_k \in F$. 
Suppose $A \in \binom{[m_{j-1} - h_j - 1]}{m_j} \oplus F_j$ (taking $m_0 = n + 1$). 
If there is any $A_i \in \binom{[m_{j'-1} - h_{j'} - 1]}{m_{j'}} \oplus F_{j'}$ where $j' < j$, then $|A_i| \geq m_{j'} \geq m_{j-1} > |A|$, a contradiction.
For any $A_i \in \binom{[m_{j'-1} - h_{j'} - 1]}{m_{j'}} \oplus F_{j'}$ where $j' > j$, we have $A_i \cap [m_{j-1} - h_j, m_{j-1}-1] = \emptyset$.
In order that $\left ( \cup_{i=1}^k A_i \right ) \cap  [m_{j-1} - h_j - 1] = A \cap [m_{j-1} - h_j-1]$,
there must be at least one $A_i \in \binom{[m_{j-1} - h_j - 1]}{m_j} \oplus F_j$ with $A_i \cap [m_{j-1} - h_j - 1] = A \cap [m_{j-1}-h_j - 1]$. 
But since $A_i \ne A$, we know for these $A_i$'s that $A_i \cap [m_{j-1} - h_j, m_{j-1}-1] \ne A \cap [m_{j-1} - h_j, m_{j-1}-1]$. 
Since $F_j$ is union-free, it is impossible to have  $\left ( \cup_{i=1}^k A_i \right ) \cap  [m_{j-1} - h_j, m_{j-1}-1] = A \cap [m_{j-1} - h_j, m_{j-1}-1]$, a contradiction. 
\end{proof}

Although the families in $\hat Q(n)$ are union-free, they are not necessarily maximal:

\begin{eg} \thlabel{cushioning not maximal}
\rm{
To the family
\[ 
q(5; 1,3,\{ \emptyset, \{ 3\}, \{ 3, 4\}, \{ 3, 5\}, \{4, 5\} \} ) 
=  \binom{[2]}{1} \oplus \{ \emptyset, \{ 3\}, \{ 3, 4\}, \{ 3, 5\}, \{4, 5\} \} 
\]
we can add the subset $\{ 3, 4, 5 \}$ 
so that the resulting family is still union-free. 
} \end{eg}

In general, is ``cushioning'' useful for getting larger union-free families? 
For now, assume that the maximum size of a  union-free family of subsets of $[n]$ is close to $\binom{n}{\ceil{n/2}}$. 
By \thref{stirling}, we compute that 
\[ \binom{n}{\ceil{n/2}} \simeq \frac{1}{\sqrt{2 \pi}} \cdot \frac{2^{n+1}}{\sqrt n}.\]
Suppose we take a ``cushion'' of ``thickness'' $h$ in the first step.
Considering the largest union-free families of subsets of $[n-t]$ and of $[n-t+1, n]$, and invoking the assumption, we see that the maximum size of the resulting family is roughly 
\[ \frac{1}{\sqrt{2 \pi}} \cdot \frac{2^{n-t +1}}{\sqrt n-t } \cdot \frac{1}{\sqrt{2 \pi}} \cdot \frac{2^{t+1}}{\sqrt t} = \frac{1}{2 \pi} \cdot \frac{2^{n+2}}{\sqrt{t(n-t)}} 
= \frac{1}{\sqrt{2 \pi}} \cdot \frac{2^{n+1}}{\sqrt n} \cdot \frac{1}{\sqrt{2 \pi}} \cdot 2\sqrt{\frac{n}{t(n-t)}}\]
which is smaller than 
$\frac{1}{\sqrt{2 \pi}} \cdot \frac{2^{n+1}}{\sqrt n}$ in general. 
Hence it seems that ``cushioning'' does not help much in getting larger union-free families. 
However, the above argument does not hold if its assumption is false, that is, if we find union-free families of subsets of $[n]$ whose size is substantially larger than $\binom{n}{\ceil{n/2}}$ (which is what ``cushioning'' may enable us to do). 

In the next section we will look at a general theory that can be used to increment the families we have seen so far.

\section{A general theory}

In this section we discuss a general theory that can be used to construct union-free families of subsets.

\begin{thm} \thlabel{general} 
If $F_1, F_2, ..., F_p, G_1, G_2, ..., G_p$ are union-free families of subsets of $[n]$, possibly including the empty set, such that
\begin{itemize} 
\item  $\left ( \bigcup_{j=1}^p \bigcup_{A \in F_j} A \right ) \cap \left ( \bigcup_{j=1}^p \bigcup_{B \in G_j} B \right ) =\emptyset $, and 
\item if $i < j$, then for any $A_1 \in F_i$ and $A_2 \in F_j$ we have $A_1 \subsetneq A_2$, and for any $B_1 \in G_i$ and $B_2 \in G_j$ we have $B_2 \subsetneq B_1$, 
\end{itemize} 
then the family 
\[ \bigcup_{j=1}^p (F_j \oplus G_j)\]
is union-free (and possibly includes the empty set). 
\end{thm}

\begin{proof}
Let $F = \bigcup_{j=1}^p \bigcup_{A \in F_j} A $
and $G = \bigcup_{j=1}^p \bigcup_{B \in G_j} B$.
Then we have $F \cap G = \emptyset$. 

Suppose there is some $j$ and some $A_0 \in F_j$ and $B_0 \in G_j$ such that $S_0 = A_0 \cup B_0$ is the union of some other subsets $S_1, S_2, ..., S_k$ in the above family. 
If $S_1, ... S_k$ are all in $F_j \oplus G_j$, write each $S_r$ as $A_r \cup B_r$ where $A_r \in F_j$ and $B_r \in G_j$. 
Then $A_0 = \bigcup_{r=1}^k A_r$ and $B_0 = \bigcup_{r =1}^k B_r$, both of which are impossible because $F_j$ and $G_j$ are union-free. 

Thus there must be some $S_r$ that is in $F_i \oplus G_i$ for some $i \ne j$. Again, write $S_r = A_r \cup B_r$ where $A_r \in F_i$ and $B_r \in G_i$. 
If $i < j$, then by the second bullet point above, $S_0 \cap G = B_0 \subsetneq B_r = S_r \cap G$, a contradiction (since $S_r$ should be a proper subset of $S_0$).
Similarly, if $i > j$, then $S_0 \cap F = A_0 \subsetneq A_r = S_r \cap  F$, a contradiction.
\end{proof}

Now recall \thref{cushioning not maximal} in the previous section: 
\[ \begin{split}
&q(5; 1,3,\{ \emptyset, \{ 3\}, \{ 3, 4\}, \{ 3, 5\}, \{4, 5\} \} ) \cup \{ \{ 3, 4, 5\} \} \\
&=\left [  \binom{[2]}{1} \oplus \{ \emptyset, \{ 3\}, \{ 3, 4\}, \{ 3, 5\}, \{4, 5\} \} \right ] \cup \{ \{3, 4, 5\} \}.
\end{split}\]
This can be written as $(F_1 \oplus G_1) \cup (F_2 \oplus G_2 )$
where
\[ F_1 = \{ \emptyset \},\]
\[ F_2 = \{ \{ 1\}, \{ 2\} \},\]
\[ G_1 = \{ \{ 3, 4, 5\}\}\]
and
\[ G_2 = \{ \emptyset, \{ 3\}, \{ 3, 4\}, \{3, 5 \}, \{ 4, 5\} \}.\]

\section{Bounds}

In \thref{general}, taking $p = 2$ and $F_2 = G_1 = \{ \emptyset \}$, we obtain the following bound on the largest size $M(n)$ of a union-free family of non-empty subsets of $[n]$:
\[ M(n) + 1 \geq \max_{1 \leq h \leq n-1}  [ M(h)+1] + [ M(n-h)+1],\]
which yields \[ M(n) \geq \max_{1 \leq h \leq n-1}  M(h) + M(n-h)+ 1.\]

Also, by \thref{cushioning free}, we have 
 \begin{align*}
M(n)  \geq & \max \Bigg \{ \binom{n-h_1}{m_1} [M(h_1) + 1] + \binom{m_1 - h_2 - 1}{m_2} [M(h_2) + 1 ] \\
& + \cdots + \binom{m_l - h_l - 1}{m_l} [M(h_l)+1] \Bigg \} ,
\end{align*}
where the maximum is taken over all values satisfying the specifications in \thref{cushioning def}. 

In particular, looking at $q(n)$, we obtain a convenient lower bound: 
\[ M(n) \geq \binom{n}{\ceil{n/2}} \simeq \sqrt{\frac{2}{\pi}} \cdot \frac{2^n}{n^{1/2}}.\]

Next we establish an upper bound on $M(n)$.

\begin{thm} \thlabel{upper 1}
For any positive integers $n, n_1, n_2$ with $n = n_1 + n_2$, 
\[ M(n) \leq M(n_1) + 2^{n_1} \cdot M(n_2).\]
\end{thm}

\begin{proof}
Suppose we are to construct a union-free family of subsets of $[n]$. 
Among the subsets of $[n]$ that are also subsets of $[n_1]$, at most $M(n_1)$ of them can be chosen. 
As for the subsets of $[n]$ that are not subsets of $[n_1]$, each of them can be uniquely written as $A \cup B$ for some $A \subset [n_1]$ and $B \in [n_1+1, n]$, $B \ne \emptyset$. 
For each particular $A \subset [n_1]$, at most $M(n_2)$ such subsets $A \cup B$ can be chosen. 
Therefore the total number of subsets in the union-free family is at most $M(n_1) + 2^{n_1} \cdot M(n_2)$. 
\end{proof}

An application of this theorem yields the following result: 

\begin{thm}\thlabel{upper 2}
For any positive integers $c$ and $k$, 
\[ M(ck) \leq \left (2^{ck}-1 \right ) \cdot \frac{M(k)}{2^k-1 }.\]
\end{thm}

\begin{proof}
We proceed by induction on $c$. When $c = 1$, the statement is trivial.
For $k > 1$, by \thref{upper 1} and the inductive hypothesis, we have
\[ \begin{split}
M(ck) & \leq M((c-1)k) + 2^{(c-1)k} \cdot M(k) \\
& \leq \left ( 2^{(c-1)k }-1  \right ) \cdot \frac{M(k)}{2^k -1} + 2^{(c-1)k} \cdot M(k)\\
& =(2^{ck}-1) \cdot \frac{M(k)}{2^k-1}.\qedhere
\end{split}\]
\end{proof}

Note that \thref{upper 2} gives an upper bound of $M(n)/(2^n-1)$ as a constant, while the bound $M(n) \geq \sqrt{\frac{2}{\pi}} \cdot \frac{2^n}{n^{1/2}}$ gives $M(n)/(2^n-1)$ as a constant multiple of $n^{-1/2}$. 
Hence the upper bound is only a polynomial multiple of the lower bound. 

Since the number of possible families of nonempty subsets of $[n]$ is $2^{2^n-1}$ which grows at a double exponential rate, it is difficult to find $M(n)$ by direct searching. 
However, with the help of the above results and computer programming, we can  establish some bounds on $M(n)$ for particular values of~$n$. 

In the following table, L.B. stands for lower bound and U.B. stands for upper bound. In the proofs of the upper bounds, when we state the values of $n_1$ and $n_2$, we mean that  the proof is by applying \thref{upper 1} with those values of $n_1$ and $n_2$. Details of the proof (by exhaustion) of the upper bound for $n = 4$ is given in the Appendix. 

{\small
\begin{center}
\renewcommand{\arraystretch}{2}
\begin{longtable}{| c || c | c || c | c | | c |}
\hline
$n$ & L.B. & Example & U.B. & Proof  & $\frac{\text{U.B.}}{\text{L.B.}}$\\ \hline
\endhead
1 & 1 & $q(1) = \binom{[1]}{1} $ & 1 & By exhaustion & 1.00\\ \hline 
2 & 2 & $q(2) = \binom{[2]}{1} $& 2 & By exhaustion & 1.00 \\ \hline
3 &  4& $q(3) = \binom{[3]}{2} \cup \binom{[1]}{1}$& 4& By exhaustion &1.00\\ \hline
4 & 7 & $q(4) = \binom{[4]}{2} \cup \binom{[1]}{1}$ & 7 &  By exhaustion &  1.00\\ \hline
5 & 13 & $q(5; 2, 1, \{ \emptyset, \{ 5 \} \}; 1, 0, \{ \emptyset \} )$ & 15 & $n_1 = 1, n_2 = 4$ & 1.15\\ \hline
6 & 22 & $q(6) = \binom{[6]}{3} \cup \binom{[3]}{2} \binom{[1]}{1}$ & 30 & $n_1 = 2 , n_2 = 4$ & 1.36 \\ \hline
7 & 39 & $q(7) = \binom{[7]}{4} \cup \binom{[3]}{2} \cup \binom{[1]}{1}$ & 60& $n_1=3,n_2=4$ &1.54 \\ \hline
8 & 74 & $q(8) = \binom{[8]}{4} \cup \binom{[3]}{2} \cup \binom{[1]}{1}$ & 119 & $n_1=4,n_2=4$ & 1.61 \\ \hline
9 & 133& $q(9) = \binom{[9]}{5} \cup \binom{[4]}{2}\cup \binom{[1]}{1}$ &  239 & $n_1=1,n_2=8$& 1.80 \\ \hline
10 & 259 & $q(10) = \binom{[10]}{5} \cup \binom{[4]}{2} \cup \binom{[1]}{1}$& 478& $n_1=2,n_2=8 $ &1.85 \\ \hline
11 & 474 & $q(11) = \binom{[11]}{6} \cup \binom{[5]}{3} \cup \binom{[2]}{1}$ &  956& $ n_1=3,n_2=8$& 2.02
\\ \hline
12 &  936 & $q(12) = \binom{[12]}{6} \cup \binom{[5]}{3} \cup \binom{[2]}{1}$ & 1911 & $n_1=4,n_2=8$ & 2.04 \\ \hline
13 & 1738 & $q(13) = \binom{[13]}{7} \cup \binom{[6]}{3} \cup \binom{[2]}{1}$ & 3823 & $n_1=1,n_2=12$ & 2.20 \\ \hline
14 &3454 & $q(14)= \binom{[14]}{7} \cup \binom{[6]}{3} \cup \binom{[2]}{1}$ & 7646&  $n_1=2,n_2=12$ & 2.21 \\ \hline
15 & 6474 & $q(15) = \binom{[15]}{8} \cup \binom{[7]}{4} \cup \binom{[3]}{2} \cup \binom{[1]}{1}$ &15292 & $n_1=3,n_2=12$ & 2.36 \\ \hline
16 & 12909 & $q(16) = \binom{[16]}{8} \cup \binom{[7]}{4} \cup \binom{[3]}{2} \cup \binom{[1]}{1}$ & 30583& $n_1=4,n_2=12$ &  2.37 \\ \hline
17 & 24384 & $q(17) = \binom{[17]}{9} \cup \binom{[8]}{4} \cup \binom{[3]}{2} \cup \binom{[1]}{1}$ &61167 & $n_1=1,n_2=16$ &  2.51 \\ \hline
18 & 48694 & $q(18) = \binom{[18]}{9} \cup \binom{[9]}{5} \cup \binom{[4]}{2}\cup \binom{[1]}{1}$ & 122334 & $n_1=2,n_2=16$ &  2.51 \\ \hline
19 & 92511 & $q(19) = \binom{[19]}{10} \cup \binom{[9]}{5} \cup \binom{[4]}{2}\cup \binom{[1]}{1}$ &244668 &$n_1=3,n_2=16$& 2.64 \\ \hline
20 & 184889& $q(20) = \binom{[20]}{10} \cup \binom{[9]}{5} \cup \binom{[4]}{2}\cup \binom{[1]}{1}$ & 489335 & $n_1=4,n_2=16$ & 2.65 \\ \hline
21 & 352975 & $q(21) = \binom{[21]}{11} \cup \binom{[10]}{5} \cup \binom{[4]}{2} \cup \binom{[1]}{1}$ & 978671 & $n_1=1,n_2=20$ & 2.77 \\ \hline
22 & 705691 & $q(22) = \binom{[22]}{11} \cup \binom{[10]}{5} \cup \binom{[4]}{2} \cup \binom{[1]}{1}$& 1957342 &$n_1=2,n_2=20$ & 2.77 \\ \hline
23 & 1352552 & $q(23) = \binom{[23]}{12} \cup \binom{[11]}{6} \cup \binom{[5]}{3} \cup \binom{[2]}{1}$ &3914684 &$n_1=3,n_2=20$ &2.89 \\ \hline
24 & 2704630& $q(24) = \binom{[24]}{12} \cup \binom{[11]}{6} \cup \binom{[5]}{3} \cup \binom{[2]}{1}$ & 7829367 &$n_1=4,n_2=20$ & 2.89 \\ \hline
25 & 5201236 & $q(25) = \binom{[25]}{13} \cup \binom{[12]}{6} \cup \binom{[5]}{3} \cup \binom{[2]}{1}$ &15658735 &$n_1=1,n_2=24$ & 3.01 \\ \hline
26 & 10401536 &$q(26) = \binom{[26]}{13} \cup \binom{[12]}{6} \cup \binom{[5]}{3} \cup \binom{[2]}{1}$ & 31317470 & $n_1=2,n_2=24$ & 3.01 \\ \hline
27 & 20060038 &$q(27) = \binom{[27]}{14} \cup \binom{[13]}{7} \cup \binom{[6]}{3} \cup \binom{[2]}{1}$ &  62634940 &$n_1=3,n_2=24$ & 3.12 \\ \hline
28 & 40118338 & $q(28) = \binom{[28]}{14} \cup \binom{[13]}{7} \cup \binom{[6]}{3} \cup \binom{[2]}{1}$ & 125269879 &$ n_1=4,n_2=24$ & 3.12 \\ \hline
29 & 77562214 & $q(29) = \binom{[29]}{15} \cup \binom{[14]}{7} \cup \binom{[6]}{3} \cup \binom{[2]}{1}$ & 250539759 &$n_1=1,n_2=28$ & 3.23 \\ \hline
30 & 155120974 & $q(30) = \binom{[30]}{15} \cup \binom{[14]}{7} \cup \binom{[6]}{3} \cup \binom{[2]}{1}$ & 501079518 &$n_1=2,n_2=28$ & 3.23 \\ \hline
\end{longtable}
\end{center}
}

We remark that from each of the union-free families of subsets of $[n]$ we have constructed, we can obtain other union-free families by permuting the elements of $[n]$. 

We also observe that in the table above, for integers $k \geq 5$, the U.B./L.B. ratios for $M(2k)$ and $M(2k-1)$ are very close.

To get an idea of how powerful such filibustering is, suppose it takes one minute to handle each amendment (e.g. a roll call vote can be requested for every amendment). Then a filibuster on even a simple bill with 30 articles will take at least 155120974 minutes, i.e. about 300 years, if no countermeasure is taken. 

Finally, we propose some conjectures: 

\begin{conj} \thlabel{big o conj}
We have  $M(n) = O(|q(n)|)$. 
\end{conj}
This conjecture suggests that the lower bound $q(n)$ for $M(n)$ is tight up to a constant multiple. 

\begin{conj} \thlabel{suff large conj 1}
For any positive integer $N$, there exists a positive integer $n > N$ such that $M(n) = |q(n)|$. 
\end{conj}
This conjecture suggests that there exists arbitrarily large positive integers $n$ for which $M(n) = |q(n)|$. 

\begin{conj} \thlabel{suff large conj 2}
There exists a positive integer $N$ such that $M(n) = |q(n)|$ for all positive integers $n> N$.
\end{conj}
This conjecture suggests that the lower bound $q(n)$ for $M(n)$  is tight for all sufficiently large $n$. 

Note that \thref{suff large conj 2} implies \thref{big o conj} and \thref{suff large conj 1}.

\section*{Appendix}

Here we prove that a union-free family of subsets of $[4]$ has size at most 7. Let  $F$ be such a family.  There are a number of cases: 

Case 1: $\{1, 2, 3, 4\}$ is in  $F$. Then there is at most one subset of size 3 in $F$. 

Case 1.1: There is one subset of size 3 in $F$. WLOG let it be $\{ 1, 2, 3\}$. Then $\{ 1, 4\}$, $\{ 2, 4\}$ and $\{ 3, 4\}$ cannot be in $F$. Among the remaining subsets, at most one of $\{1, 2\}$ and $\{3 \}$ can be in $F$, and at most one of $\{1, 3\}$ and $\{2 \}$ can be in $F$. Hence, there are at most 7 subsets in $F$. 

Case 1.2: There is no subset of size 3 in $F$. 

Case 1.2.1: There is at least one subset of size 2 in $F$. WLOG let it be $\{ 1, 2\}$. Then $\{ 3, 4\}$ cannot be in $F$, at most one of $\{ 1\}$ and $\{ 2\}$ can be in $F$, at most one of $\{ 3 \}$ and $\{ 4\}$ can be in $F$, and at most one of $\{ 1,4 \}$ and $\{ 2, 3\}$ can be in $F$. Hence there are at most 7 subsets in $F$. 

Case 1.2.2: There is no subset of size 2 in $F$. Then there are at most 5 subsets in $F$. 

Case 2: $\{ 1, 2, 3, 4\}$ is not in $F$. 

Case 2.1: All subsets of size 3 are in $F$. 

Case 2.1.1: There is at least one subset of size 2 in $F$. WLOG let it be $\{1, 2\}$. Then $\{ 1, 3\}$, $\{ 1, 4\}$, $\{ 2, 3\}$ and $\{ 2, 4\}$  cannot be in $F$. If $\{ 3, 4\}$ is in $F$, then at most one subset of size 1 can be in $F$. If $\{ 3, 4\}$ is not in $F$, then at most 2 subsets of size 1 can be in $F$. In either case, there are at most 7 subsets in $F$. 

Case 2.1.2: There is no subset of size 2 in $F$. Then as there are at most 2 subsets of size 1 in $F$, it follows that there are at most 6 subsets in $F$. 

Case 2.2: Exactly 3 subsets of size 3 are in $F$. WLOG let them be $\{ 1, 2, 3\}$, $\{ 1, 2, 4\}$ and $\{ 1, 3, 4\}$. 

Case 2.2.1: At least one subset of size 2 is in $F$.

Case 2.2.1.1: A subset of size 2 that contains the element 1 is in $F$. WLOG let it be $\{ 1, 2\}$. Then $\{ 1, 3\}$, $\{ 1, 4\}$, $\{ 2, 3\}$ and $\{ 2, 4\}$  cannot be in $F$. If $\{ 3, 4\}$ is in $F$, then at most 2 subsets of size 1 can be in $F$. If $\{ 3, 4\}$ is not in $F$, then at most 3 subsets of size 1 can be in $F$. In either case, there are at most 7 elements in $F$. 

Case 2.2.1.2: No subset of size 2 that contains the element 1 is in $F$, but a subset of size 2 that does not contain the element 1 is in $F$. WLOG suppose $\{ 2, 3\}$ is in $F$. Then $\{ 1\}$ cannot be in $F$, and it is easy to check that no 4 of the remaining subsets can be in $F$. Hence there are at most 7 subsets in $F$. 

Case 2.2.2: No subset of size 2 is in $F$. Then there are at most 7 subsets in $F$. 

Case 2.3: Exactly 2 subsets of size 3 are in $F$. WLOG let them be $\{ 1, 2, 3\}$ and $\{ 1, 2, 4\}$. 

Case 2.3.1: $\{ 1, 2\}$ is in $F$. Then $\{1, 3\}$, $\{1, 4\}$, $\{ 2, 3\}$ and $\{ 2, 4\}$ cannot be in $F$. Also, at most one of $\{1 \}$ and $\{ 2\}$ can be in $F$. Hence there are at most 7 subsets in $F$. 

Case 2.3.2: $\{ 1, 2\}$ is not in $F$. Note that at most one of $\{ 1, 3\}$ and $\{ 2, 3\}$ can be in $F$, and at most one of $\{ 1, 4\}$ and $\{ 2, 4\}$ can be in $F$. After checking all possibilities, we find that there are at most 7 subsets in $F$. 

Case 2.4: Exactly 1 subset of size 3 is in $F$. WLOG let it be $\{ 1, 2, 3\}$. 

Case 2.4.1: One of $\{ 1, 2\}$, $\{ 1, 3\}$ and $\{ 2, 3\}$ is in $F$. WLOG assume $\{ 1, 2\}$ is in $F$. Then $\{ 1, 3\}$ and $\{ 2, 3\}$ cannot be in $F$. Note that at most one of $\{ 1\}$ and $\{ 2\}$ can be in $F$, and at most two of $\{ 3, 4\}$, $\{3\}$ and $\{4 \}$ can be in $F$. Hence there are at most 7 subsets in $F$. 

Case 2.4.2: None of  $\{ 1, 2\}$, $\{ 1, 3\}$ and $\{ 2, 3\}$ is in $F$. Clearly at most two of $\{1\}$, $\{2\}$ and $\{3\}$ can be in $F$. Hence there are at most 7 subsets in $F$. 

Case 2.5: No subset of size 3 is in $F$. 

Case 2.5.1: All 4 subsets of size 1 are in $F$. Then there can be no subset of size 2 in $F$ and so there are only 4 subsets in $F$.

Case 2.5.2: Exactly 3 subsets of size 1 are in $F$. WLOG let them be $\{ 1\}$, $\{2\}$ and $\{ 3\}$. Then the only subsets of size 2 that can be in $F$ are $\{ 1, 4\}$, $\{ 2, 4\}$ and $\{3, 4\}$. Hence there are at most 6 subsets in $F$. 

Case 2.5.3: Exactly 2 subsets of size 1 are in $F$. WLOG let them be $\{ 1\}$ and $\{ 2\}$. Then $\{ 1, 2\}$ cannot be in $F$. Hence there are at most 7 subsets in $F$. 

Case 2.5.4: Exactly 1 or 0 subset of size 1 is in $F$. Then there are at most 7 subsets in $F$. 

In conclusion, the size of the family $F$ cannot exceed 7.

\section*{Acknowledgements}
The author is deeply indebted to Professor Richard Ehrenborg for his help with this \mbox{paper}.

\section*{About the Student Author}

Andy Loo is an undergraduate student in Princeton University's Class of 2016. His research interests include number theory, combinatorics, and game theory. 

\section*{Press Summary}

This paper discusses the question of how many non-empty subsets of a given finite set we can choose so that no chosen subset is the union of some other chosen subsets.
This question informs the issue of how many amendments to a bill a legislator can propose in order to filibuster against it. 
It is found that a bill with as few as 30 articles can theoretically face a filibuster that lasts about 300 years. 

\end{document}